\documentclass[12pt,psamsfonts]{amsart}
\usepackage{amsmath,amsthm,amsfonts,amssymb, enumitem}
\usepackage{eucal}
\usepackage{graphicx}
\usepackage[all,knot]{xy}
\xyoption{arc}

\newcommand{\CC}{\mathcal{C}}
\newcommand{\ep}{\epsilon}

\newcommand{\F}{\mathbb{F}}

\newcommand{\Rep}{{\rm Rep}}

\newcommand{\so}{\mathfrak{so}}
\newcommand{\N}{\mathbb{N}}
\newcommand{\Q}{\mathbb Q}
\newcommand{\R}{\mathbb{R}}

\DeclareMathOperator{\Aut}{Aut}
\DeclareMathOperator{\GL}{GL}

\DeclareMathOperator{\SL}{SL}

\DeclareMathOperator{\PSO}{PSO}
 \DeclareMathOperator{\Gal}{Gal}

\newcommand{\C}{\mathbb C}
\newcommand{\mC}{\mathcal{C}}

\newcommand{\Z}{\mathbb Z}

\newcommand{\comments}[1]{}

\renewcommand{\CC}{\mathcal{C}}

\renewcommand{\Z}{\mathbb{Z}}

\numberwithin{equation}{section}

\newtheorem{theorem}{Theorem}[section]

\newtheorem{lem}[theorem]{Lemma}

\newtheorem{prop}[theorem]{Proposition}
\theoremstyle{definition}

\newtheorem{remark}[theorem]{Remark}

\newtheorem{definition}[theorem]{Definition}

\begin{document}
\title[]{Classification of Rank 6 Modular Categories with Galois Group $\langle (012)(345)\rangle$}

\author{David Green}
\email{green.2116@buckeyemail.osu.edu}
\address{Department of Mathematics\\The Ohio State University
    \\Columbus, Ohio\\
    U.S.A}

\thanks{The author was supported by NSF grant DMS-1757872}

\begin{abstract}

Modular Tensor Categories (MTC’s) arise in the study of certain condensed matter systems. There is an ongoing program to classify MTC’s of low rank, up to modular data. We present an overview of the methods to classify modular tensor categories of low rank, applied to the specific case of a rank 6 category with Galois group $\langle(012)(345)\rangle$, and show that certain symmetries in this case imply nonunitarizable (hence, nonphysical) MTC’s. We show that all the rank 6 MTC's with this Galois group have modular data conjugate to either the product of the semion category with $(A_1, 5)_{\frac{1}{2}}$ or a certain modular subcategory of $\CC(\so_5, 9, e^{j\pi i/9})$ with gcd$(18, j) = 1$. 
\end{abstract}
\maketitle

\section{Introduction} The proof of rank-finiteness for modular categories \cite{BNRW1} allows for a classification-by-rank program. This paper uses a combination of Galois theory and computational algebraic geometry to extend the classification to the rank 6 case, where the Galois group is $\langle(012)(345)\rangle$. The primary leverage is provided by Lemma \ref{nondegen}, which allows application of a classification of fusion rules. A list of all possible Galois groups in the rank 6 case is available in a recent thesis by Daniel Creamer \cite{C}, and efforts to complete the classification by Galois group are well underway. 

A complete classification is known through rank 5. For ranks $\leq 4$, the classification was completed in \cite{RSW}, and the rank 5 case is available in \cite{BNRW2}. In the weakly integral case the classification is completed through rank 7 in \cite{BGNPRW}. Focusing on the non-integral rank $6$ case, Creamer's thesis determines all the possible Galois groups in both the self-dual and non-self-dual  cases, and furthermore completes the classification problem in the non-self-dual case. Furthermore, the only example with $\Gal(\mC)  = \langle (012)\rangle$ is known to be the modular data available in \cite{S}, by an unpublished result of Creamer's. Combined with the results of this paper, the only remaining groups in the rank $6$ case are the following groups: 

$$\langle(012345)\rangle,  \quad \text{ and }\quad \langle(01)(23)(45), (02)(13)\rangle.$$
These last cases may be potentially more difficult than the case at hand, since here we leverage the fact that $|\Gal(\mC)| = 3$ is a prime power. 
Additionally, this result shows that the fusion rules of $\PSO(5)_{3/2}$ constructed in \cite{R} are the smallest rank fusion rules with no unitary realization. 
\section{Preliminaries}

\noindent In this section, we recall notation, definitions, and results, mostly from \cite{BNRW1} and \cite{BNRW2}. \\

\noindent A \textbf{modular category} $\mC$ is a braided spherical fusion category with an invertible $S$-matrix. We will recall only the relevant details, and refer the reader to \cite{BNRW1} for a more complete description. Let $|\Pi_\mC|$ be the finite set of isomorphism classes of $\mC$.  In a rank $r$ modular category, we will label the isomorphism classes of simple objects with a label in $0, 1, \dots, r- 1$, and let $V_i$ be in the isomorphism class of label $i$. $\mC$ comes equipped with an involution $*$ on the set of isomorphism classes of simple objects satisfying $0^* = 0$. 

\begin{definition}
We say that $\mC$ is self dual (SD) if $V_{i*} \cong V_i$ for all isomorphism classes of simple objects, and non-self-dual (NSD) otherwise. 
\end{definition}

\noindent The \textbf{fusion rules} $N_{i,j}^k$ are defined by means of the decomposition $$V_i \otimes V_j \cong \bigoplus_{k = 0}^{r - 1}N_{i,j}^k V_k$$ and then assembled into the $r$ matrices $N_k$ where $(N_i)_{k,j} = N_{i,j}^k \in \N$. Following convention, every matrix in this paper will be $0$-indexed. 

For a complex matrix $A$, we define $\F_A$ as the smallest field containing the entries of $A$. Also, for an integer $M$, we let $\Q_M = \Q\left(e^{2\pi i/M}\right)$. For any Galois extension $\F/\Q$ and $\sigma \in \Gal(\F/\Q)$, If $A$ has entries in $\F$, $\sigma(A)$ is obtained from $A$ by applying $\sigma$ entry-wise. 

A modular data $(S, T)$ is said to be \textbf{realizable} if there exists a modular category with modular data $(S, T)$. 
We reproduce the following definition and theorem from \cite{BNRW1}(2.14, 2.15): 
\begin{definition} For $S,T \in \GL_r(\C)$, define the constants $d_j := S_{0j}, \theta_j := T_{jj},\\ D^2 := \sum_j d_j^2$ and $p_\pm = \sum_{k = 0}^{r -1}d_k^2\theta_k^{\pm1}$. The pair $(S,T)$ is called \textbf{admissible} if
\begin{enumerate}[label = (\roman*)]
    \item $d_j \in \R$ and $S = S^t$ with $SS^\dagger = D^2I$. $T$ is diagonal with finite order $N$. 
    \item $(ST)^3 = p_+S^2, p_+p_- = D^2$, and $p^+/p_-$ is a root of unity.
    \item $N_{i,j}^k := \frac{1}{D^2}\sum_{a = 0}^{r - 1}\frac{S_{ia}S_{ja}\overline{S_{ka}}}{S_{0a}} \in \N$ for all $0 \leq i,j,k \leq (r-1)$
    \item $\theta_i\theta_jS_{ij} = \sum_{k=0}^{r-1}N_{i^*j}^kd_id_j\theta_k$, where $i^*$ is the unique label such that $N^{0}_{i, i*} = 1$. 
    \item Define $\nu_n(k) := \frac{1}{D^2}\sum_{i,j = 0}^{r -1} N^k_{i,j}d_id_j\left(\frac{\theta_i}{\theta_j}\right)^n$. Then $\nu_2(k)= 0$ if $k \ne k^*$ and $\nu_2(k) = \pm1$ if $k^* = k$. For all $n, k$, $\nu_n(k) \in \Z\left[e^{2\pi i/N}\right]$.
    \item  $\F_S \subset \F_T = \Q_N$,$\Gal(\F_S/\Q)$ is isomorphic to an abelian subgroup of $\mathfrak{S}_r$ and $\Gal(\F_T/\F_S) \cong (\Z/2\Z)^l$ for some $l \in \N$.
    \item The prime divisors of the principal ideals generated by $D^2$ and $N$ coincide in $\Z\left[e^{2\pi i/N}\right]$. 
\end{enumerate}
\begin{definition}
If a modular category $\mC$ realizes data such that the $d_i$ are all integers, we say $\mC$ is \textit{integral}. If $d_i^2$ is an integer for all $i$, $\mC$ is said to be \textit{weakly integral}. 
\end{definition}
\begin{theorem} \cite{BNRW1} Let $(S,T)$ be realizable modular data. Then \begin{enumerate}[label=(\alph*)]
    \item $(S, T)$ is admissible and 
    \item For all $\sigma \in \Aut(\Q_{ab}), (\sigma(S), \sigma(T))$ is realizable. We call $(\sigma(S), \sigma(T))$ a Galois conjugate of $(S, T)$. 
\end{enumerate}
\end{theorem}
\noindent Let $\tilde{S}$ be the matrix obtained by taking $S$ and dividing column $i$ by $d_i$. $\tilde{S}$ is also orthogonal, diagonalizes the $N_i$, and the $ith$ row of $\tilde{S}$ is the set of eigenvalues of $N_i$.

Part $(ii)$ of the definition of admissible modular data implies that $(S, T)$ define a projective representation of $\SL(2, \Z) \cong \langle s,t | s^2 = (st)^3 = I \rangle$. In  \cite{BNRW2} it is shown that for some $12th$ root of unit $\gamma$, the matrix $\gamma T$ is the image of $t$ in a representation of $\SL(2, \Z)$. Denote the entries of this diagonal matrix by $t_i = \gamma T_{ii}$. 
\begin{theorem}\cite{BNRW2}(2.12, 2.5, 3.1) For each $\sigma \in \Gal(S)$ 
\begin{enumerate}[label=(\alph*)]
    \item $\sigma$ permutes the columns of $\tilde{S}$, and we will abuse notation and identify $\sigma$ with this permutation, and write $\Gal(\mC)$ as generated by these permutations.  
    \item $\sigma^2(t_i) = t_{\sigma(i)}$ (Galois symmetry). 
    \item There exists a sign function $\ep \colon \Pi_\mC \to \{\pm1\}$ depending on $\sigma$ such that $$S_{ij} = \ep_\sigma(i)\ep_{\sigma^{-1}}(j)S_{\sigma(i)\sigma^{-1}(j)}$$ We call a choice of $\ep$ a \textit{sign choice}, and when $\sigma$ is clear from the context, we write $\ep_\sigma(i)$ as $\ep_i$. 
    \item When $r$ is even, $\prod_{i = 0}^{r-1}\ep_\sigma(i) = (-1)^\sigma$
\end{enumerate}
\end{theorem}
\end{definition}

\section{Classification} 
We assume from here forward that $\mC$ is a modular tensor category with $\Gal(\mC) = \langle(012)(345)\rangle$. In particular, this implies that $\mC$ is nonintegral (column zero not fixed) and self dual (no order two elements, so the $S$ matrix is real). 
\begin{lem} \label{relabel} If $\mC$ is a rank 6 non-integral, self dual, modular tensor category, there are at most $7$ sign choices (up to relabeling of the simple objects).
\end{lem}
\begin{proof}
We have at first $2^6$ sign choices. Since the rank is even, we have that $\ep_0\ep_1\ep_2\ = \ep_3\ep_4\ep_5$. Since only products $\ep_i\ep_j$ appear in the matrix, we may  choose $\ep_0 = 1$. These constraints leave us with 16 possibilities. They are as follows (we only list either the positive or the negative signs, for brevity.) 
\begin{align*}
    &0)~ \ep_i = 1 &&\textbf{1)}~ \ep_0 = \ep_1 = 1 &&&\textbf{2)}~ \ep_0 = \ep_2  = 1&&&&\textbf{3)} ~\ep_0 = \ep_3  = 1\\
      &4)~ \ep_0 = \ep_4 = 1 &&5)~ \ep_0 = \ep_5 = 1 &&&\textbf{6)}~ \ep_1 = \ep_5  = -1 &&&&\textbf{7)}~ \ep_2 = \ep_5  = -1\\
       &\textbf{8)}~ \ep_3 = \ep_5 = -1 &&9)~ \ep_4 = \ep_5 = -1 &&&10)~ \ep_1 = \ep_4  = -1 &&&&11)~ \ep_2 = \ep_4  = -1\\
        &12)~ \ep_3 = \ep_4 = -1 &&13)~ \ep_1 = \ep_3 = -1 &&&14)~ \ep_2 = \ep_3  = -1&&&&\textbf{15)}~ \ep_1 = \ep_2  = -1\\
\end{align*}
The cyclic relabeling $(345)$ respects the action of the Galois group. Applying this relabeling to the matrices given by the above sign choices, we see that cases 3,4,5 are equivalent, as are cases 6,10, 13, cases 8,9,12, and cases 7, 11, 14. A chosen representative from each equivalence class for the purposes of computation is bolded. Sign choice 0 can be eliminated by noticing that for every column in $S$ there is another column that has the same entries rearranged.  The column of FP dimensions in $\tilde{S}$ will be all positive, so that $\tilde{S}$ is not orthogonal in this case. 
\end{proof}
\begin{lem} \label{nondegen}Suppose $\rho$ is a representation of $\SL(2, \Z)$ defined by modular data of a rank 6, self-dual MTC with $\Gal(\mC) = \langle(012)(345)\rangle$. Then $\rho$ is irreducible. 
\end{lem}
\begin{proof} By (\cite{EHR},Lemma 1), it suffices to show that $\rho$ has nondegenerate $t$-spectra. We use the Galois symmetries of both the $S$ and $T$ matrices to reduce this problem to thirty-five Gr\"obner basis computations, which are included in Section \ref{Grobner}.  By the preceding lemma, the final list of sign choices to be checked for nondegeneracy by Gr\"obner basis algorithm is  $1,2,3,6,7,8,15$. 
There is some symmetry to exploit before beginning the calculations, grouping the types of degeneracies that could potentially occur into 5 cases. These are: 
\begin{enumerate}
    \item $1 = \theta_0 = \theta_1 = \theta_2$
    \item $\theta_3 = \theta_4 = \theta_5$
    \item $\theta_0 = \theta_3, \theta_1 = \theta_4, \theta_2 = \theta_5$
    \item $\theta_0 = \theta_4, \theta_1 = \theta_5, \theta_2 = \theta_3$
    \item $\theta_0 = \theta_5, \theta_1 = \theta_3, \theta_2 = \theta_4$
\end{enumerate} 
We see this as follows: Suppose $\theta_i = \theta_j$. Then $t_i = t_j$, and $t_{\sigma(i)} = \sigma^2(t_i) = \sigma^2(t_j) = t_{\sigma(j)}$ by Galois symmetry, implying $\theta_{\sigma(i)} = \theta_{\sigma(j)}$ Repeating this process gives a third equation. The five cases are thus the two cases where both $i$ and $j$ are both in $\{0,1,2\}$ or $\{3,4,5\}$, and the three cases where $i \in \{0,1,2\}$ and $j \in \{3,4,5\}$. 
The  Gr\"obner basis computations included in Section \ref{Grobner} show that in all 5 cases we have a contradiction, so the spectra is nondegenerate.
\end{proof}
\begin{lem} \label{factor} Suppose $\mathcal{C}$ is a non-integral modular fusion category of rank 6, with $K(\mathcal{C}) \cong K_1 \otimes K_2$ with $K_i$ nontrivial. Then $\mathcal{C} \cong B_1 \boxtimes B_2$ for $B_i$ modular.
\end{lem}
\begin{remark}
The analogue of this lemma is not true in rank 4, with the toric code forming a counterexample (The fusion rules are a product, but the category is not. See \cite{RSW} 5.3.8.), so at least one of the rank/non-integrality assumptions is necessary.  
\end{remark}
\begin{proof}
WLOG, Let $K_1$ have rank 2, and $K_2$ rank 3. Let $B_1$ and $B_2$ be the associated full fusion subcategories associated with $K_1$ and $K_2$. If either $B_1$ or $B_2$ is modular, \cite{M}(Theorem 4.2) we have the result. So we assume $B_1$ not modular and show $B_2$ is. If $B_1$ is not modular, the simple objects have the same dimension as $\Rep(\Z_2)$, and are thus integral(\cite{O1}, 2.4). We may assume that $B_2$ is not symmetric since by \cite{D} this implies integrality of $B_2$ and thus $\mC$. Then \cite{O2} gives that the fusion rules are either Ising or $\Rep(S_3)$. We eliminate $\Rep(S_3)$ since it again implies integrality of $B_2$. But all MTC's with Ising fusion rules are modular by \cite{DGNO}(Appendix B), so $B_2$ is modular. 
\end{proof}

\begin{theorem} Up to relabeling, and Galois conjugation, the only realizable modular data for rank 6, non-integral, self dual, MTC's with Galois group $\langle(012)(345)\rangle$ are given by the following two families of $(S, T)$. 
    $$S = \begin{bmatrix} 1 & 1 \\ 1 & -1 \end{bmatrix} \otimes \begin{bmatrix} 1 & d & d^2 - 1 \\ d & -(d^2 - 1) & 1 \\ d^2 - 1 & 1 & -d \end{bmatrix} ~~,~~T = \begin{bmatrix} 1 & \\ & i \end{bmatrix} \otimes \begin{bmatrix}1 & & \\& e^{2\pi i /7}&\\&& e^{10\pi i /7}\end{bmatrix}$$
    where $d = 2\cos{(\pi/7)}$ and 
    $$S =
\begin{bmatrix}
 1 & -1 & 1 & r_1 & r_2 & r_3 \\
 -1 & 1 & -1 & -r_2 & -r_3 & -r_1 \\
 1 & -1 & 1 & r_3 & r_1 & r_2 \\
 r_1 & -r_2 & r_3 & 1 & 1 & 1 \\
 r_2 & -r_3 & r_1 & 1 & 1 & 1 \\
 r_3 & -r_1 & r_2 & 1 & 1 & 1 \\
\end{bmatrix}$$$$T = \begin{bmatrix}
 1 \\ & e^{2\pi i/3} \\ && e^{4\pi i/3} \\ &&& e^{-4\pi i /9}\\ &&&& e^{8\pi i /9} \\ &&&&& e^{2\pi i /9} \end{bmatrix}$$
where with $\alpha = e^{i\pi/9}$ we set $r_1 = -\alpha - \alpha^2 + \alpha^5, r_2 = \alpha + \alpha^2 - \alpha^4$ and $r_3 = \alpha^4 - \alpha^5$.
\end{theorem}

\begin{proof}
By \ref{nondegen}, the representation of $\SL(2, \Z)$ coming from $\mC$ is irreducible, so that the representation is (up to a character) either the tensor product of two irreducible representations of dimensions 2 and 3, or a 6 dimensional irreducible representation. We know either 7 or $9$ divides N by \cite{BNRW2}, Proposition 3.13.
\begin{itemize}[leftmargin=*]\item If $7|N$, the representation factors non-trivially through $\SL(2, \Z/7\Z)$ and Table 12 in \cite{EHR} indicates there is only one such irreducible representation, with dimension 3. Lemma 3 from \cite{EHR} and Lemma \ref{factor} show that $\mC$ factors, and the previous classification of rank 2 and 3 MTC's \cite{RSW}(Theorem 3.2) implies that the first pair of modular data is the only solution. The realization is as the product of the semion category and $(A_1, 5)_{\frac{1}{2}}$. These categories are defined in \cite{RSW} in sections 5.3.1 and 5.3.6, respectively.
\item If $9|N$ then the only representation is of type $B_9$, up to a (real) character. This fusion algebra is known to be related to quantum groups coming from $\so$ which fits with the construction for the categories in this case. The notation is from \cite{EHR} and seems to also appear in the physics literature. We relabel the fusion rules given in \cite{EHR}(Appendix B)  by the permutation $(14)(253)$ and give the fusion matrix $N_4$ (Eholzer's $N_1$) as:
$$\left(
\begin{array}{cccccc}
 0 & 0 & 0 & 0 & 1 & 0 \\
 0 & 1 & 1 & 1 & 0 & 1 \\
 0 & 1 & 1 & 0 & 1 & 0 \\
 0 & 1 & 0 & 1 & 0 & 0 \\
 1 & 0 & 1 & 0 & 0 & 1 \\
 0 & 1 & 0 & 0 & 1 & 0 \\
\end{array}
\right)$$
This determines the algebra since the $N_i$ commute. We can then compute the characteristic polynomials for the $N_i$, $n_i(x)$. While in principle the fusion matrix $N_1$ determines the others, if one happens to have the $S$-matrix for a MTC with given fusion rules, (see \cite{R}) it is much easier to use the Verlinde formula to compute the $N_i$.
\begin{align*}
    n_0(x) &= (x - 1)^6 \\ n_1(x) &= (x + 1)^3(x^3 - 6x^2 + 3x + 1)\\  
    n_2(x) &= (x - 1)^3(x^3 - 3x^2 - 6x - 1)\\
    n_3(x) = n_4(x) = n_5(x) &= (x^3 - 3x  + 1)(x^3 - 3x^2 + 1)
\end{align*}
Let $q_1\ q_2, q_3$ be the  irrational roots of $n_1(x)$ and $t_1, t_2, t_3$ be the irrational roots of $n_2(x)$. Let $r_1, r_2, r_3$ be roots of one factor of $n_3(x)$. Note that the roots of the other factor are the inverses of the $r_i$. and permuted the same way. Now, we will construct the $\tilde{S}$ matrix, row by row, and show by contradiction that the second row must be $(-1,-1,-1, q_1, q_2, q_3)$. By the Galois action, we know the only other possibilities for the first three rows of  $\tilde{S}$ are 
$$\begin{bmatrix}
1 & 1 & 1 & 1 & 1 & 1 \\
q_1 & q_2 & q_3 & -1 & -1 & -1 \\
1 & 1 & 1 & t_1 & t_2 & t_3 \\
        \end{bmatrix} \quad \text{and} \quad \begin{bmatrix}1 & 1 & 1 & 1 & 1 & 1 \\
q_1 & q_2 & q_3 & -1 & -1 & -1 \\
t_1 & t_2 &t_3 & 1 & 1 & 1 \\
        \end{bmatrix}$$ In the first case we see that $S_{12} = q_1 = \pm 1$, an immediate contradiction. In the second, we get that $q_1 = \pm t_2^{-1}$, which implies that the minimial polynomial of $q_1$ is either a factor of $n_2$ or obtained by changing signs on these factors, a contradiction. Likewise, if we choose the third row to begin with the $t_i$, we get that $-t_2 = \pm1$. Now we may solve for the top left 3x3 submatrix of $S$. If two rows start with the same ordering of the $r_i$, the inner product of the first two columns isn't a symmetric function of the roots, hence not $3$ and the $S$ matrix isn't orthogonal. Now we may observe the Galois group action on $\tilde{S}$ is respected only if the $S$ has the following form: 
        $$\begin{bmatrix}
 1 & -1 & 1 & r_1 & r_2 & r_3 \\
 -1 & 1 & -1 & -r_2 & -r_3 & -r_1 \\
 1 & -1 & 1 & r_3 & r_1 & r_2 \\
 r_1 & -r_2 & r_3 \\
 r_2 & -r_3 & r_1 \\
 r_3 & -r_1 & r_2 \\
\end{bmatrix}$$ 
This tells us a few things. $r_1r_2 + r_2r_3 + r_1r_2 = -3$, which means that the minimal polynomial for the $r_i$ is $x^3 - 3x + 1$, so that $r_1 + r_2 + r_3 = 0$. Taking the inner product with the top three rows gives the following system of equations,
\begin{align*}
    r_1S_{33} + r_2S_{34} + r_3S_{35} &= 0 \\
    r_2S_{33} + r_3S_{34} + r_1S_{35} &= 0 \\
    r_3S_{33} + r_1S_{34} + r_2S_{35} &= 0 
\end{align*}
which solve to yield $S_{33} = S_{34} = S_{35}$. Then $\sigma(S_{33}/r_1) = S_{33}/r_2$, and we can conclude that all the $S_{ij}$ are 1, which gives the stated solution set for $S$. Computing the norm of the first column gives $D^2 = 9$. By \cite{ETF}(Theorem 5.1), we know that $N|D^5$ which implies the character from the beginning of the proof is trivial. The $T$ matrix is determined by use of Gr\"obner Bases, see the end of Section \ref{Grobner}. We get the relations:  $$\theta_2^2 + \theta_2 + 1, \theta_1 + \theta_2 + 1,\theta_3 + \theta_4 + \theta_5, \theta_2\theta_4 - \theta_5, x^2 + x + 1, \theta_5^2 + x+\theta_5 - d_5$$, where $x = (\theta_1\theta_2\theta_3\theta_4\theta_5)^{-1}$, included to force the $\theta_i$ nonzero in the calculations. Including the equations with $x$ also aids computer algebra systems in enumerating all twelve solutions. At this point we may calculate the second Frobenius-Schur indicators $\nu_2(k)$ for each label $k$ in all twelve solutions, 96 in total. Since all particles are self-dual, we require $\nu_2(k) = \pm 1$ for all $k$. The only two solutions are the given solution and its complex conjugate. This gives us a total of six $T$-spectra up to relabeling, all of which are Galois conjugates of each other.
Modular categories with the remaining 6 $T$-spectra are constructed in \cite{R} as subcategories of $\CC(\so_5, 9, e^{j\pi i/9})$ with gcd$(18, j) = 1$ and are thus realizeable. This completes the classification.

\end{itemize}
\end{proof}

\section{Acknowledgements}
The author would like to thank Dr.\ Eric Rowell for supervising this project during the 2018 Mathematics REU at Texas A\&M university, as well as Daniel Creamer for providing helpful explanations and direction. 

\section{Future Work} Related open problems include the complete classification of rank 6 MTC's by Galois group, as well as MTC's of higher rank. In the classification process, it seemed that the symmetries of sign choice 15 alone implied non-unitarizablilty of any MTC realizing them. Additionally, Lemma 3.3 is expected to hold in higher generality, and investigating such a claim might help in the classification of higher rank categories. Finally, in the $7|N$ case of the classification we know all of the actual categories that realize the modular data. Is it possible to determine the categories in the $9|N$ case? 

\section{Gr\"obner Basis Calculations} \label{Grobner}
The algorithm used to find contradictions in modular data with degenerate $t$-spectra is essentially unchanged from \cite{C}. For each sign choice, we initialize the ideal to be generated by the orthogonality and twist relations, as well the relation $d_1d_2d_3d_4d_5k - 1= 0$ to force dimensions nonzero. Then we add the sign relations from the case. After each calculation of a Gr\"obner basis, we search the factored output for new relations to add to the ideal, and recalculate the basis. As an example, if the polynomial $\theta_3(d_1 - 1)$ occurs in the output, we would add the relation $d_1 - 1$ to the ideal and rerun, since $\theta_3 \ne 0$. If the resulting ideal is the unit ideal, we mark 1 in the ``Factored Polynomials" column and continue to the next case. 
All the computations were performed on a Dell XPS 9560, using Macaulay2 version 1.1. It is possible to instruct Macaulay to compute a basis for only the relations of degree less than $n$. This speeds up computations drastically, and we include $n$ (which may change during the computation) in the data given. The first set of computations are from Lemma \ref{nondegen}, and the last one is the relations on the $T$ matrix from the end of the classification. 

Creamer noted in \cite{C} that when the Galois group was a subgroup of one that is realizable, it was harder to eliminate it. Likewise, when the sign choice is one actually occurs, it's harder to show that the $t$-spectra is nondegenerate. Compare sign choice 2 and sign choice 15. 

\underline{\textit{Calculations for Lemma \ref{nondegen}:}} We process the 7 sign choices in increasing numerical order, first eliminating the case $\theta_1 = \theta_2 = 1$ (case 1), and then eliminate the $\theta_3 = \theta_4 = \theta_5$ case (case 2) while assuming $\theta_1$ and $\theta_2$ are both distinct and not 1. Then, we finish the remaining three cases, which often reduce to case 1 or 2. For reference, the cases are: 
\begin{enumerate}
    \item $1 = \theta_0 = \theta_1 = \theta_2$
    \item $\theta_3 = \theta_4 = \theta_5$
    \item $\theta_0 = \theta_3, \theta_1 = \theta_4, \theta_2 = \theta_5$
    \item $\theta_0 = \theta_4, \theta_1 = \theta_5, \theta_2 = \theta_3$
    \item $\theta_0 = \theta_5, \theta_1 = \theta_3, \theta_2 = \theta_4$
\end{enumerate} 

The contradictions obtained require the use of the following:
\begin{itemize}
    \item $D, p,  \theta_i$ are all nonzero.
    \item $\Q(S) \subset \Q(T)$
    \item If $\theta_{\sigma(i)} = \theta_i$, then $\theta_i = \theta_{\sigma(i)} = \theta_{\sigma^2(i)}$
\end{itemize}
We will use the following without comment:
\begin{prop} In case 1, 
$\theta_4, \theta_5 \not \in \{\pm 1\}$.  
\end{prop}
\begin{proof}
Suppose otherwise. In this case $t_4/t_5 = \pm 1$, so that $\pm 1 = \sigma^2(t_4/t_5) = t_5/t_3$ and thus $\theta_3 = \pm 1$. But then $\F_S \subset \F_T = \Q$, a contradiction. 
\end{proof}
\noindent\textit{Sign Choice 1, Case 1:} Initial ideal generated by twist relations, orthogonality, $\theta_1 - 1, \theta_2 - 1$.
\begin{center}
\begin{tabular}{|c|c|c|}\hline
 \text{Degree Limit}
     & \text{Factored Polynomials} 
     & \text{Zero Factors Added}
     \\
     \hline
     9 & $p(d_3 + d_4 + d_5)D$ & $d_3 + d_4 + d_5$ \\
     \hline 
     9 & $D^4(\theta_5^2 - 1),D^4(\theta_4^2 - 1)$ & \\
     \hline 
\end{tabular}
\end{center}
\noindent\textit{Sign Choice 1, Case 2:} Initial ideal generated by twist relations, orthogonality, $\theta_3 - \theta_4, \theta_4 - \theta_5$.
\begin{center}
\begin{tabular}{|c|c|c|}\hline
 \text{Degree Limit}
     & \text{Factored Polynomials} 
     & \text{Zero Factors Added}
     \\
     \hline
     8 & $p(\theta_5 + 1)(\theta_1 - \theta_2)$ & $\theta_5 + 1$ \\
     \hline 
     8 & $D(\theta_1 + \theta_2), (p + 2)D$ & $p+2, \theta_1 + \theta_2$ \\ \hline 8 & 1 &  \\
     \hline
\end{tabular}
\end{center}
~\\
\noindent\textit{Sign Choice 1, Case 3:} Initial ideal generated by twist relations, orthogonality, $\theta_3 - 1, \theta_4 - \theta_1, \theta_5 - \theta_2$.
\begin{center}
\begin{tabular}{|c|p{1.75in}|p{1.75in}|}\hline
 \text{Degree Limit}
     & \text{Factored Polynomials} 
     & \text{Zero Factors Added}
     \\
     \hline
     7 & $p(S_{35}\theta_4 + S_{34}\theta_5 + S_{33})$ \newline$p(d_5\theta_4 + d_4\theta_5 + d_3)$\newline $p(d_2\theta_4 + d_1\theta_5 - 1)$ & $S_{35}\theta_4 + S_{34}\theta_5 + S_{33}$\newline $d_5\theta_4 + d_4\theta_5 + d_3$ \newline$d_2\theta_4 + d_1\theta_5 - 1$ \\
     \hline 
     7 &  $(\theta_4 + \theta_5)d_3D^2$ \newline $(\theta_5)(\theta_5 +1)d_5D^2$ & $\theta_4 + \theta_5$ \newline $\theta_5 + 1$ \\ \hline
\end{tabular}
\end{center}
We now have $\theta_4 = \theta_1 = 1$, so we reduce to case 1 and are done. 
~\\
\noindent\textit{Sign Choice 1, Case 4:} Initial ideal generated by twist relations, orthogonality, $\theta_3 - \theta_2, \theta_4 - 1, \theta_5 - \theta_1$.
\begin{center}
\begin{tabular}{|c|p{1.75in}|p{1.75in}|}\hline
 \text{Degree Limit}
     & \text{Factored Polynomials} 
     & \text{Zero Factors Added}
     \\
     \hline
     7 & $p(S_{33}\theta_3 + S_{34}\theta_5 + S_{35})$ \newline$p(d_5\theta_3 + d_3\theta_5 + d_4)$\newline $p(d_1\theta_3 + d_2\theta_5 - 1)$ & $S_{33}\theta_3 + S_{34}\theta_5 + S_{35}$\newline $d_5\theta_3 + d_3\theta_5 + d_4$ \newline$d_1\theta_3 + d_2\theta_5 - 1$ \\
     \hline 
     7 &  $(\theta_3 + \theta_5)d_4D^2$ \newline $(\theta_5)(\theta_5 +1)d_5D^2$ & $\theta_3 + \theta_5$ \newline $\theta_5 + 1$ \\ \hline
\end{tabular}
\end{center}
This gives $\theta_3 = \theta_2 = 1$, so we reduce to case 1 and are done. 
~\\
\noindent\textit{Sign Choice 1, Case 5:} Initial ideal generated by twist relations, orthogonality, $\theta_3 - \theta_1, \theta_4 - \theta_2, \theta_5 - 1$.
\begin{center}
\begin{tabular}{|c|p{1.75in}|p{1.75in}|}\hline
 \text{Degree Limit}
     & \text{Factored Polynomials} 
     & \text{Zero Factors Added}
     \\
     \hline
     7 & $p(S_{33}\theta_3 + S_{35}\theta_4 + S_{34})$ \newline$p(d_4\theta_3 + d_3\theta_4 + d_5)$\newline $p(d_2\theta_3 + d_1\theta_4 - 1)$ & $S_{33}\theta_3 + S_{35}\theta_4 + S_{34}$\newline $d_4\theta_3 + d_3\theta_4 + d_5$ \newline$d_2\theta_3 + d_1\theta_4 - 1$ \\
     \hline 
     7 &  $(\theta_3 + \theta_4)d_5D^2$ \newline $(\theta_4)(\theta_4 +1)d_4D^2$ & $\theta_3 + \theta_4$ \newline $\theta_4 + 1$ \\ \hline
\end{tabular}
\end{center}
At this point, we have $\theta_3 = \theta_1 = 1$, so we reduce to case 1 and are done. 
~\\
%DONE%
\noindent\textit{Sign Choice 2:} All five cases immediately returned the unit ideal. Degree limits used were 9 for cases 1-2, and 7 for 3-5. 
~\\
\textit{Sign Choice 3, Case 1:} Initial ideal generated by twist relations, orthogonality, $\theta_1 - 1, \theta_2 - 1$.
\begin{center}
\begin{tabular}{|c|c|c|}\hline
 \text{Degree Limit}
     & \text{Factored Polynomials} 
     & \text{Zero Factors Added}
     \\
     \hline
     9 & $p(d_3 - d_4 + d_5)D$ & $d_3 -d_4 + d_5$ \\
    \hline 
    9 & $D^2(p^2 - D^2), D^4(\theta_3 + \theta_4 + \theta_5 + 1)$ & $p^2 - D^2, \theta_3 + \theta_4 + \theta_5 + 1$ \\ 
     \hline 
     9 & $D^4(\theta_5^2 - 1),D^4(\theta_4^2 - 1)$ & \\
     \hline 
\end{tabular}
\end{center}
~\\
\textit{Sign Choice 3, Case 2:} Initial ideal generated by twist relations, orthogonality, $\theta_3 - \theta_4, \theta_4 - \theta_5$.
\begin{center}
\begin{tabular}{|c|c|c|}\hline
 \text{Degree Limit}
     & \text{Factored Polynomials} 
     & \text{Zero Factors Added}
     \\
     \hline
     8 & $(\theta_5 + 1)(\theta_1 - \theta_2)pD$ & $\theta_5 + 1$ \\
     \hline
     9 & $(\theta_2 - 1)^3p$ & \\
     \hline
\end{tabular}
\end{center}
~\\
\noindent\textit{Sign Choice 3, Case 3:} Initial ideal generated by twist relations, orthogonality, $\theta_3 - 1, \theta_4 - \theta_1, \theta_5 - \theta_2$.
\begin{center}
\begin{tabular}{|c|p{1.75in}|p{1.75in}|}\hline
 \text{Degree Limit}
     & \text{Factored Polynomials} 
     & \text{Zero Factors Added}
     \\
     \hline
     7 & 
     $p(S_{35}\theta_4 - S_{34}\theta_5 + S_{33})$ 
     \newline$p(d_5\theta_4 - d_4\theta_5 + d_3)$
     \newline $p(d_2\theta_4 - d_1\theta_5 + 1)$ 
     & $S_{35}\theta_4 - S_{34}\theta_5 + S_{33}$
     \newline $d_5\theta_4 - d_4\theta_5 + d_3$ 
     \newline$d_2\theta_4 - d_1\theta_5 + 1$ \\
     \hline 
     7 &  $(\theta_4 + \theta_5)d_3D^2$ \newline $(\theta_5)(\theta_5 +1)d_5D^2$ & $\theta_4 + \theta_5$ \newline $\theta_5 + 1$ \\ \hline
\end{tabular}
\end{center}
We get $\theta_4 = \theta_1 = 1$, so we reduce to case 1 and are done. 
~\\
\noindent\textit{Sign Choice 3, Case 4:} Initial ideal generated by twist relations, orthogonality, $\theta_3 - \theta_2, \theta_4 - 1, \theta_5 - \theta_1$.
\begin{center}
\begin{tabular}{|c|p{1.75in}|p{1.75in}|}\hline
 \text{Degree Limit}
     & \text{Factored Polynomials} 
     & \text{Zero Factors Added}
     \\
     \hline
     7 & $p(S_{33}\theta_3 - S_{34}\theta_5 + S_{35})$ \newline$p(d_5\theta_3 + d_3\theta_5 - d_4)$\newline $p(d_1\theta_3 - d_2\theta_5 - 1)$ & $S_{33}\theta_3 - S_{34}\theta_5 + S_{35}$\newline $d_5\theta_3 + d_3\theta_5 - d_4$ \newline$d_1\theta_3 - d_2\theta_5 - 1$ \\
     \hline 
     7 &  $(\theta_3 + \theta_5)d_4D^2$ \newline $(\theta_5)(\theta_5 +1)d_5D^2$ & $\theta_3 + \theta_5$ \newline $\theta_5 + 1$ \\ \hline
\end{tabular}
\end{center}
The relation $\theta_3 = \theta_2 = 1$ reduces us to case 1 and are done. 
~\\
\noindent\textit{Sign Choice 3, Case 5:} Initial ideal generated by twist relations, orthogonality, $\theta_3 - \theta_1, \theta_4 - \theta_2, \theta_5 - 1$.
\begin{center}
\begin{tabular}{|c|p{1.75in}|p{1.75in}|}\hline
 \text{Degree Limit}
     & \text{Factored Polynomials} 
     & \text{Zero Factors Added}
     \\
     \hline
     7 & $p(S_{33}\theta_3 + S_{35}\theta_4 - S_{34})$ 
     \newline$p(d_4\theta_3 - d_3\theta_4 - d_5)$
     \newline $p(d_2\theta_3 - d_1\theta_4 + 1)$ 
     & $S_{33}\theta_3 + S_{35}\theta_4 - S_{34}$
     \newline $d_4\theta_3 - d_3\theta_4 - d_5$ \newline$dd_2\theta_3 - d_1\theta_4 + 1$ \\
     \hline 
     7 &  $(\theta_3 + \theta_4)d_5D^2$ \newline $(\theta_4)(\theta_4 +1)d_5D^2$ & $\theta_3 + \theta_5$ \newline $\theta_5 + 1$ \\ \hline
\end{tabular}
\end{center}
We can conclude $\theta_3 = \theta_1 = 1$, so we reduce to case 1 and are done. 
~\\
\textit{Sign Choice 6, Case 1:} Initial ideal generated by twist relations, orthogonality, $\theta_1 - 1, \theta_2 - 1$.
\begin{center}
\begin{tabular}{|c|c|c|}\hline
 \text{Degree Limit}
     & \text{Factored Polynomials} 
     & \text{Zero Factors Added}
     \\
     \hline
     9 & $p(d_3 - d_4 -d_5)D$ & $d_3 -d_4 - d_5$ \\
    \hline 
     9 & $D^4(\theta_5^2 - 1),D^4(\theta_4^2 - 1)$ & \\
     \hline 
\end{tabular}
\end{center}
~\\
\textit{Sign Choice 6, Case 2:} Identical to sign choice 3, case 2.
~\\
\noindent\textit{Sign Choice 15, Case 3:} Initial ideal generated by twist relations, orthogonality, $\theta_3 - 1, \theta_4 - \theta_1, \theta_5 - \theta_2$.
\begin{center}
\begin{tabular}{|c|p{1.75in}|p{1.75in}|}\hline
 \text{Degree Limit}
     & \text{Factored Polynomials} 
     & \text{Zero Factors Added}
     \\
     \hline
     7 & $p(S_{35}\theta_4 + S_{34}\theta_5 - S_{33})$ 
     \newline$p(d_5\theta_4 + d_4\theta_5 - d_3)$
     \newline $p(d_2\theta_4 - d_1\theta_5 - 1)$ 
     & $S_{35}\theta_4 + S_{34}\theta_5 - S_{33}$
     \newline $d_5\theta_4 + d_4\theta_5 - d_3$ 
     \newline$d_2\theta_4 - d_1\theta_5 - 1$ \\
     \hline 
     7 &  $(\theta_4 + \theta_5)d_3D^2$ 
     \newline $p(\theta_5 +1)d_5D^2$ 
     & $\theta_4 + \theta_5$ \newline $\theta_5 + 1$ \\ \hline
\end{tabular}
\end{center}
So $\theta_4 = \theta_1 = 1$, and we reduce to case 1 and are done. 
~\\
\noindent\textit{Sign Choice 15, Case 4:} Initial ideal generated by twist relations, orthogonality, $\theta_3 - \theta_2, \theta_4 - 1, \theta_5 - \theta_1$.
\begin{center}
\begin{tabular}{|c|p{1.75in}|p{1.75in}|}\hline
 \text{Degree Limit}
     & \text{Factored Polynomials} 
     & \text{Zero Factors Added}
     \\
     \hline
     7 & $p(S_{33}\theta_3 - S_{34}\theta_5 - S_{35})$
     \newline$p(d_5\theta_3 - d_3\theta_5 + d_4)$
     \newline $p(d_1\theta_3 - d_2\theta_5 + 1)$ 
     & $S_{33}\theta_3 - S_{34}\theta_5 - S_{35}$
     \newline $d_5\theta_3 - d_3\theta_5 + d_4$ 
     \newline$d_1\theta_3 - d_2\theta_5 + 1$ \\
     \hline 
     7 &  $(\theta_3 + \theta_5)d_4D^2$ \newline $(\theta_5)(\theta_5 +1)d_5D^2$ & $\theta_3 + \theta_5$ \newline $\theta_5 + 1$ \\ \hline
\end{tabular}
\end{center}
We now have $\theta_3 = \theta_2 = 1$, so we reduce to case 1 and are done. 
~\\
\noindent\textit{Sign Choice 15, Case 5:} Initial ideal generated by twist relations, orthogonality, $\theta_3 - \theta_1, \theta_4 - \theta_2, \theta_5 - 1$.
\begin{center}
\begin{tabular}{|c|p{1.75in}|p{1.75in}|}\hline
 \text{Degree Limit}
     & \text{Factored Polynomials} 
     & \text{Zero Factors Added}
     \\
     \hline
     7 & $p(S_{33}\theta_3 - S_{35}\theta_4 - S_{34})$ 
     \newline$p(d_4\theta_3 - d_3\theta_4 + d_5)$
     \newline $p(d_2\theta_3 - d_1\theta_4 - 1)$ 
     & $S_{33}\theta_3 - S_{35}\theta_4 - S_{34}$
     \newline $d_4\theta_3 - d_3\theta_4 + d_5$ 
     \newline$d_2\theta_3 - d_1\theta_4 - 1$ \\
     \hline 
     7 &  $(\theta_3 + \theta_4)d_5D^2$ \newline $(\theta_4)(\theta_4 +1)d_4D^2$ & $\theta_3 + \theta_4$ \newline $\theta_4 + 1$ \\ \hline
\end{tabular}
\end{center}
At this point, we have $\theta_3 = \theta_1 = 1$, so we reduce to case 1 and are done. 
~\\
%DONE%
\noindent\textit{Sign Choice 7, Case 1:} Identical to case 6
\noindent\textit{Sign Choice 7, Case 2:} Identical to case 6
~\\
~\\
\noindent\textit{Sign Choice 7, Case 3:} Initial ideal generated by twist relations, orthogonality, $\theta_3 - 1, \theta_4 - \theta_1, \theta_5 - \theta_2$.
\begin{center}
\begin{tabular}{|c|p{1.75in}|p{1.75in}|}\hline
 \text{Degree Limit}
     & \text{Factored Polynomials} 
     & \text{Zero Factors Added}
     \\
     \hline
     7 & $p(S_{35}\theta_4 + S_{34}\theta_5 - S_{33})$ 
     \newline$p(d_5\theta_4 + d_4\theta_5 - d_3)$
     \newline $p(d_2\theta_4 + d_1\theta_5 - 1)$ 
     & $S_{35}\theta_4 + S_{34}\theta_5 - S_{33}$
     \newline $d_5\theta_4 + d_4\theta_5 - d_3$ 
     \newline$d_2\theta_4 + d_1\theta_5 - 1$ \\
     \hline 
     7 &  $(\theta_4 + \theta_5)d_3D^2$ 
     \newline $p(\theta_5 +1)d_5D^2$ 
     & $\theta_4 + \theta_5$ \newline $\theta_5 + 1$ \\ \hline
\end{tabular}
\end{center}
At this point, we have $\theta_4 = \theta_1 = 1$, so we reduce to case 1 and are done. 
~\\
\noindent\textit{Sign Choice 15, Case 4:} Initial ideal generated by twist relations, orthogonality, $\theta_3 - \theta_2, \theta_4 - 1, \theta_5 - \theta_1$.
\begin{center}
\begin{tabular}{|c|p{1.75in}|p{1.75in}|}\hline
 \text{Degree Limit}
     & \text{Factored Polynomials} 
     & \text{Zero Factors Added}
     \\
     \hline
     7 & $p(S_{33}\theta_3 - S_{34}\theta_5 - S_{35})$
     \newline$p(d_5\theta_3 - d_3\theta_5 + d_4)$
     \newline $p(d_1\theta_3 + d_2\theta_5 - 1)$ 
     & $S_{33}\theta_3 - S_{34}\theta_5 - S_{35}$
     \newline $d_5\theta_3 - d_3\theta_5 + d_4$ 
     \newline$d_1\theta_3 + d_2\theta_5 - 1$ \\
     \hline 
     7 &  $(\theta_3 + \theta_5)d_4D^2$ \newline $(\theta_5)(\theta_5 +1)d_5D^2$ & $\theta_3 + \theta_5$ \newline $\theta_5 + 1$ \\ \hline
\end{tabular}
\end{center}
At this point, we have $\theta_3 = \theta_2 = 1$, so we reduce to case 1 and are done. 
~\\
\noindent\textit{Sign Choice 15, Case 5:} Initial ideal generated by twist relations, orthogonality, $\theta_3 - \theta_1, \theta_4 - \theta_2, \theta_5 - 1$.
\begin{center}
\begin{tabular}{|c|p{1.75in}|p{1.75in}|}\hline
 \text{Degree Limit}
     & \text{Factored Polynomials} 
     & \text{Zero Factors Added}
     \\
     \hline
     7 & $p(S_{33}\theta_3 - S_{35}\theta_4 - S_{34})$ 
     \newline$p(d_4\theta_3 - d_3\theta_4 + d_5)$
     \newline $p(d_2\theta_3 + d_1\theta_4 - 1)$ 
     & $S_{33}\theta_3 - S_{35}\theta_4 - S_{34}$
     \newline $d_4\theta_3 - d_3\theta_4 + d_5$ 
     \newline$d_2\theta_3 + d_1\theta_4 - 1$ \\
     \hline 
     7 &  $(\theta_3 + \theta_4)d_5D^2$ \newline $(\theta_4)(\theta_4 +1)d_4D^2$ & $\theta_3 + \theta_4$ \newline $\theta_4 + 1$ \\ \hline
\end{tabular}
\end{center}
Thus $\theta_3 = \theta_1 = 1$, so we reduce to case 1 and are done. 
~\\
%DONE%
\noindent\textit{Sign Choice 8, Case 1:} Identical to sign choice 6, case 1
~\\
\noindent\textit{Sign Choice 8, Case 2:} Identical to sign choice 6, case 2
~\\
\noindent\textit{Sign Choice 8, Case 3:} Initial ideal generated by twist relations, orthogonality, $\theta_3 - 1, \theta_4 - \theta_1, \theta_5 - \theta_2$.
\begin{center}
\begin{tabular}{|c|p{1.75in}|p{1.75in}|}\hline
 \text{Degree Limit}
     & \text{Factored Polynomials} 
     & \text{Zero Factors Added}
     \\
     \hline
     7 & $p(S_{35}\theta_4 - S_{34}\theta_5 - S_{33})$ 
     \newline$p(d_5\theta_4 - d_4\theta_5 - d_3)$
     \newline $p(d_2\theta_4 + d_1\theta_5 + 1)$ 
     & $S_{35}\theta_4 - S_{34}\theta_5 - S_{33}$
     \newline $d_5\theta_4 - d_4\theta_5 - d_3$ 
     \newline$d_2\theta_4 + d_1\theta_5 + 1$ \\
     \hline 
     7 &  $(\theta_4 + \theta_5)d_3D^2$ 
     \newline $(\theta_5)(\theta_5 +1)d_5D^2$ 
     & $\theta_4 + \theta_5$ \newline $\theta_5 + 1$ \\ \hline
\end{tabular}
\end{center}
At this point, we have $\theta_4 = \theta_1 = 1$, so we reduce to case 1 and are done. 
~\\
\noindent\textit{Sign Choice 8, Case 4:} Initial ideal generated by twist relations, orthogonality, $\theta_3 - \theta_2, \theta_4 - 1, \theta_5 - \theta_1$.
\begin{center}
\begin{tabular}{|c|p{1.75in}|p{1.75in}|}\hline
 \text{Degree Limit}
     & \text{Factored Polynomials} 
     & \text{Zero Factors Added}
     \\
     \hline
     7 & $p(S_{33}\theta_3 + S_{34}\theta_5 - S_{35})$
     \newline$p(d_5\theta_3 - d_3\theta_5 - d_4)$
     \newline $p(d_1\theta_3 + d_2\theta_5 + 1)$ 
     & $S_{33}\theta_3 + S_{34}\theta_5 - S_{35}$
     \newline $d_5\theta_3 - d_3\theta_5 - d_4$ 
     \newline$d_1\theta_3 + d_2\theta_5 + 1$ \\
     \hline 
     7 &  $(\theta_3 + \theta_5)d_4D^2$ \newline $(\theta_5)(\theta_5 +1)d_5D^2$ & $\theta_3 + \theta_5$ \newline $\theta_5 + 1$ \\ \hline
\end{tabular}
\end{center}
This gives $\theta_3 = \theta_2 = 1$, so we reduce to case 1 and are done. 
~\\
\noindent\textit{Sign Choice 8, Case 5:} Initial ideal generated by twist relations, orthogonality, $\theta_3 - \theta_1, \theta_4 - \theta_2, \theta_5 - 1$.
\begin{center}
\begin{tabular}{|c|p{1.75in}|p{1.75in}|}\hline
 \text{Degree Limit}
     & \text{Factored Polynomials} 
     & \text{Zero Factors Added}
     \\
     \hline
     7 & $p(S_{33}\theta_3 - S_{35}\theta_4 + S_{34})$ 
     \newline$p(d_4\theta_3 + d_3\theta_4 - d_5)$
     \newline $p(d_2\theta_3 + d_1\theta_4 + 1)$ 
     & $S_{33}\theta_3 - S_{35}\theta_4 + S_{34}$
     \newline $d_4\theta_3 + d_3\theta_4 - d_5$ 
     \newline$d_2\theta_3 + d_1\theta_4 + 1$ \\
     \hline 
     7 &  $(\theta_3 + \theta_4)d_5D^2$ \newline $(\theta_4)(\theta_4 +1)d_4D^2$ & $\theta_3 + \theta_4$ \newline $\theta_4 + 1$ \\ \hline
\end{tabular}
\end{center}
At this point, we have $\theta_3 = \theta_1 = 1$, so we reduce to case 1 and are done. 
~\\
%DONE%
 \textit{Sign Choice 15, Case 1:} Initial ideal generated by twist relations, orthogonality, $\theta_1 - 1, \theta_2 - 1$.
\begin{center}
\begin{tabular}{|c|c|c|}\hline
 \text{Degree Limit}
     & \text{Factored Polynomials} 
     & \text{Zero Factors Added}
     \\
     \hline
     8 & $p(d_3 + d_4 + d_5)D$ & $d_3 +d_4 + d_5$ \\
    \hline 
     9 & $D^4(\theta_5^2 - 1),D^4(\theta_4^2 - 1)$ & \\
     \hline 
\end{tabular}
\end{center}
~\\
\noindent\textit{Sign Choice 15, Case 2:} Identical to sign choice 6, case 2
~\\
\noindent\textit{Sign Choice 15, Case 3:} Initial ideal generated by twist relations, orthogonality, $\theta_3 - 1, \theta_4 - \theta_1, \theta_5 - \theta_2$.
\begin{center}
\begin{tabular}{|c|p{1.75in}|p{1.75in}|}\hline
 \text{Degree Limit}
     & \text{Factored Polynomials} 
     & \text{Zero Factors Added}
     \\
     \hline
     7 & $p(S_{35}\theta_4 + S_{34}\theta_5 + S_{33})$ 
     \newline$p(d_5\theta_4 + d_4\theta_5 + d_3)$
     \newline $p(d_2\theta_4 - d_1\theta_5 + 1)$ 
     & $S_{35}\theta_4 + S_{34}\theta_5 + S_{33}$
     \newline $d_5\theta_4 + d_4\theta_5 + d_3$ 
     \newline$d_2\theta_4 - d_1\theta_5 + 1$ \\
     \hline 
     7 &  $(\theta_4 + \theta_5)d_3D^2$ 
     \newline $p(\theta_5 +1)d_5D^2$ 
     & $\theta_4 + \theta_5$ \newline $\theta_5 + 1$ \\ \hline
\end{tabular}
\end{center}
We get $\theta_4 = \theta_1 = 1$, reduce to case 1, and are done. 
~\\
\noindent\textit{Sign Choice 15, Case 4:} Initial ideal generated by twist relations, orthogonality, $\theta_3 - \theta_2, \theta_4 - 1, \theta_5 - \theta_1$.
\begin{center}
\begin{tabular}{|c|p{1.75in}|p{1.75in}|}\hline
 \text{Degree Limit}
     & \text{Factored Polynomials} 
     & \text{Zero Factors Added}
     \\
     \hline
     7 & $p(S_{33}\theta_3 + S_{34}\theta_5 + S_{35})$
     \newline$p(d_5\theta_3 + d_3\theta_5 + d_4)$
     \newline $p(d_1\theta_3 - d_2\theta_5 - 1)$ 
     & $S_{33}\theta_3 + S_{34}\theta_5 + S_{35}$
     \newline $d_5\theta_3 + d_3\theta_5 + d_4$ 
     \newline$d_1\theta_3 - d_2\theta_5 - 1$ \\
     \hline 
     7 &  $(\theta_3 + \theta_5)d_4D^2$ \newline $(\theta_5)(\theta_5 +1)d_5D^2$ & $\theta_3 + \theta_5$ \newline $\theta_5 + 1$ \\ \hline
\end{tabular}
\end{center}
We recover $\theta_3 = \theta_2 = 1$, so we reduce to case 1 and are done. 
~\\
\noindent\textit{Sign Choice 15, Case 5:} Initial ideal generated by twist relations, orthogonality, $\theta_3 - \theta_1, \theta_4 - \theta_2, \theta_5 - 1$.
\begin{center}
\begin{tabular}{|c|p{1.75in}|p{1.75in}|}\hline
 \text{Degree Limit}
     & \text{Factored Polynomials} 
     & \text{Zero Factors Added}
     \\
     \hline
     7 & $p(S_{33}\theta_3 + S_{35}\theta_4 + S_{34})$ 
     \newline$p(d_4\theta_3 + d_3\theta_4 + d_5)$
     \newline $p(d_2\theta_3 - d_1\theta_4 + 1)$ 
     & $S_{33}\theta_3 + S_{35}\theta_4 + S_{34}$
     \newline $d_4\theta_3 + d_3\theta_4 + d_5$ 
     \newline$d_2\theta_3 - d_1\theta_4 + 1$ \\
     \hline 
     7 &  $(\theta_3 + \theta_4)d_5D^2$ \newline $(\theta_4)(\theta_4 +1)d_4D^2$ & $\theta_3 + \theta_4$ \newline $\theta_4 + 1$ \\ \hline
\end{tabular}
\end{center}
So $\theta_3 = \theta_1 = 1$, we reduce to case 1 and are done. 
~\\
%DONE%
\underline{\textit{Calculation of T-matrix: }} Since we know the $S$ matrix (which corresponds to sign choice 15), we simply run the algorithm ($n = 20$) with initial ideal generated by the twist and orthogonality relations,the minimal polynomials for the $d_i$ and $S_{ij}$, and the equation $\theta_1\theta_2\theta_3\theta_4\theta_5x - 1 = 0$ This gives the following relations: $$\theta_2^2 + \theta_2 + 1, \theta_1 + \theta_2 + 1,\theta_3 + \theta_4 + \theta_5, \theta_2\theta_4 - \theta_5, x^2 + x + 1, \theta_5^2 + x+\theta_5 - d_5$$, which is sufficient to finish the classification.

\end{document}